\documentclass[12pt]{amsart}
\usepackage{amscd}     
\usepackage{amssymb}
\usepackage{amsmath, amsthm, graphics, dsfont}
\usepackage{xypic}     
\LaTeXdiagrams        
\usepackage[all]{xy}
\xyoption{2cell} \UseAllTwocells \xyoption{frame} \CompileMatrices
\allowdisplaybreaks[3]
\usepackage{amsfonts}
\usepackage[normalem]{ulem}
\usepackage{hyperref}
\usepackage{tikz}
\usepackage{mathtools}
\usepackage{verbatim} 
\usepackage{MnSymbol}
\usepackage{latexsym}
\usepackage{epsfig}
\usepackage{enumerate}
\usepackage{times}
\usepackage{xcolor}
\usepackage{geometry}
\usepackage{ytableau}

\newgeometry{vmargin={25mm,25mm}, hmargin={25mm,25mm}}   

\newtheorem{prop}{Proposition}[section]
\newtheorem{lem}[prop]{Lemma}

\newtheorem{cor}[prop]{Corollary}
\newtheorem{thm}[prop]{Theorem}

\numberwithin{equation}{section}

\DeclareMathOperator{\diag}{diag}

\title{Holomorphic Anomaly Equations For $[\mathbb{C}^5/\mathbb{Z}_5]$}
\author[Genlik]{Deniz Genlik}
\address{Department of Mathematics\\ Ohio State University\\ 100 Math Tower, 231 West 18th Ave. \\ Columbus,  OH 43210\\ USA}
\curraddr{Department of Mathematics\\ University of Illinois Urbana-Champaign\\ 1409 W. Green St.\\
273 Altgeld Hall\\ M/C 382\\
Urbana, IL 61801\\ USA}

\email{genlik@illinois.edu}

\author[Tseng]{Hsian-Hua Tseng}
\address{Department of Mathematics\\ Ohio State University\\ 100 Math Tower, 231 West 18th Ave. \\ Columbus,  OH 43210\\ USA}
\email{hhtseng@math.ohio-state.edu}

\begin{document}
\begin{abstract}
We prove holomorphic anomaly equations for $[\mathbb{C}^5/\mathbb{Z}_5]$.
\end{abstract}

\date{\today}

\maketitle

{\centering \em To the memory of Bumsig Kim\par}

\tableofcontents

\setcounter{section}{-1}

\section{Introduction}

The cyclic group $\mathbb{Z}_5$ acts naturally on $\mathbb{C}^5$ by letting its generator $1\in\mathbb{Z}_5$ act by multiplication by the fifth root of unity $$e^{\frac{2\pi\sqrt{-1}}{5}}.$$
This action commutes with the diagonal action of the torus $\mathrm{T}=(\mathbb{C}^*)^5$ on $\mathbb{C}^5$ and induces a $\mathrm{T}$-action on $[\mathbb{C}^5/\mathbb{Z}_5]$. Consequently $[\mathbb{C}^5/\mathbb{Z}_5]$ is a toric Deligne-Mumford stack.

This paper is concerned with $\mathrm{T}$-equivariant Gromov-Witten invariants of $[\mathbb{C}^5/\mathbb{Z}_5]$. By definition, these are the following integrals
\begin{equation}\label{eqn:GWinv}
    \int_{\left[\overline{M}_{g, n}^{\mathrm{orb}}\left(\left[\mathbb{C}^{5} / \mathbb{Z}_{5}\right], 0\right)\right]^{vir}} \prod_{i=1}^{n} \mathrm{ev}_{i}^{*}\left(\gamma_{i}\right)\psi_i^{k_i}. 
\end{equation}
Here, $ [\overline{M}_{g, n}^{\mathrm{orb}}\left(\left[\mathbb{C}^{5} / \mathbb{Z}_{5}\right], 0\right)]^{vir}$ is the ($\mathrm{T}$-equivariant) virtual fundamental class of the moduli space $\overline{M}_{g, n}^{\mathrm{orb}}\left(\left[\mathbb{C}^{5} / \mathbb{Z}_{5}\right], 0\right)$ of stable maps to $[\mathbb{C}^5/\mathbb{Z}_5]$. The classes $\psi_i\in H^2(\overline{M}_{g, n}^{\mathrm{orb}}\left(\left[\mathbb{C}^{5} / \mathbb{Z}_{5}\right], 0\right), \mathbb{Q})$ are descendant classes, for a detailed treatment see \cite{k}.  The evaluation maps
\begin{equation*}
  \mathrm{ev}_i: \overline{M}_{g, n}^{\mathrm{orb}}\left(\left[\mathbb{C}^{5} / \mathbb{Z}_{5}\right], 0\right)\to I[\mathbb{C}^5/\mathbb{Z}_5]  
\end{equation*}
 take values in the inertia stack $I[\mathbb{C}^5/\mathbb{Z}_5]$ of $[\mathbb{C}^5/\mathbb{Z}_5]$. The classes $\gamma_i\in H^*_\mathrm{T, Orb}([\mathbb{C}^5/\mathbb{Z}_5]):=H^*_\mathrm{T}(I[\mathbb{C}^5/\mathbb{Z}_5])$ are classes in the Chen-Ruan cohomology of $[\mathbb{C}^5/\mathbb{Z}_5]$. Standard references for the Chen-Ruan cohomology are \cite{alr} and \cite{cr}.

Let $$\lambda_0,\lambda_1, \lambda_2, \lambda_3,\lambda_{4}\in H^*_\mathrm{T}(\text{pt})=H^*(B\mathrm{T})$$ be the first Chern classes of the tautological line bundles of $B\mathrm{T}=(B\mathbb{C}^*)^5$. Then (\ref{eqn:GWinv}) takes value in $\mathbb{Q}(\lambda_0,\lambda_1, \lambda_2, \lambda_3,\lambda_{4})$.

Foundational treatments of orbifold Gromov-Witten theory can be found in many references. For compact target stacks, the original reference is \cite{agv}. For non-compact target stacks admitting torus actions, such as $[\mathbb{C}^5/\mathbb{Z}_5]$, one can define Gromov-Witten theory for them using virtual localization formula \cite{grpan}. In this case, their Gromov-Witten theory should be understood as certain {\em twisted} Gromov-Witten theory of stacks. Generalities on twisted Gromov-Witten theory of stacks can be found in \cite{ccit} and \cite{Tseng}.

The main results of this paper concern structures of Gromov-Witten invariants (\ref{eqn:GWinv}), formulated in terms of generating functions. The definition of inertia stacks implies that 
\begin{equation*}
    I[\mathbb{C}^5/\mathbb{Z}_5]=[\mathbb{C}^5/\mathbb{Z}_5]\cup\bigcup_{k=1}^{4} B\mathbb{Z}_5.
\end{equation*}
Let 
\begin{equation*}
\phi_0=1\in H^0_\mathrm{T}([\mathbb{C}^5/\mathbb{Z}_5]), \phi_k=1\in H^0_\mathrm{T}(B\mathbb{Z}_5), 1\leq k\leq 4.
\end{equation*}
Then $\{\phi_0,\phi_1, \phi_2, \phi_3,\phi_{4}\}$ is an additive basis of $H^*_\mathrm{T,Orb}([\mathbb{C}^5/\mathbb{Z}_5])$. The orbifold Poincar\'e dual $\{\phi^0,\phi^1, \phi^2, \phi^3,\phi^{4}\}$ of this basis is given by
\begin{equation*}
\phi^0=5\lambda_0\lambda_1\lambda_2\lambda_3\lambda_{4}\phi_0,\quad \phi^1=5\phi_{4},\quad \phi^2=5\phi_{3},\quad \phi^3=5\phi_{2}, \quad \phi^{4}=5\phi_1.
\end{equation*}
To simplify notation, in what follows we set \begin{equation*}
\phi_i\coloneqq\phi_j\quad\text{if }j\equiv{i}\mod{5}\quad\text{and}\quad\phi^i\coloneqq\phi^j\quad\text{if }j\equiv{i}\mod{5},
\end{equation*}
for all $i\geq{0}$ and $0\leq{j}\leq{4}$.   

For $\phi_{c_{1}}, \ldots,\phi_{c_{n}}\in H^{*}_{\mathrm{T,Orb}}\left(\left[\mathbb{C}^5/\mathbb{Z}_5\right]\right)$, define the Gromov-Witten potential by
\begin{equation}\label{eq:GW_Potential_phi_insertions}
\mathcal{F}_{g, n}^{\left[\mathbb{C}^{5} / \mathbb{Z}_{5}\right]}\left(\phi_{c_{1}}, \ldots, \phi_{c_{n}}\right)=\sum_{d=0}^{\infty} \frac{\Theta^{d}}{d !} \int_{\left[\overline{M}_{g, n+d}^{\mathrm{orb}}\left(\left[\mathbb{C}^{5} / \mathbb{Z}_{5}\right], 0\right)\right]^{v i r}} \prod_{i=1}^{n} \mathrm{ev}_{i}^{*}\left(\phi_{c_{i}}\right) \prod_{i=n+1}^{n+d} \mathrm{ev}_{i}^{*}\left(\phi_{1}\right)
\end{equation}
where $\Theta$ is a formal variable keeping track of the insertion $\phi_1$. The potential $\mathcal{F}_{g, n}^{\left[\mathbb{C}^{5} / \mathbb{Z}_{5}\right]}$ is regarded as a formal power series in $\Theta$.
We also use the standard double bracket notation for the Gromov-Witten potential throughout the paper: 
\begin{equation*}
\left\langle\left\langle\phi_{c_{1}}, \ldots, \phi_{c_{n}}\right\rangle\right\rangle_{g, n}^{\left[\mathbb{C}^{5} / \mathbb{Z}_{5}\right]}=\mathcal{F}_{g, n}^{\left[\mathbb{C}^{5} / \mathbb{Z}_{5}\right]}\left(\phi_{c_{1}}, \ldots, \phi_{c_{n}}\right).
\end{equation*}

The main results of this paper are differential equations for these generating functions $\mathcal{F}^{[\mathbb{C}^5/\mathbb{Z}_5]}_{g}$ for $g\geq{2}$ after the following specializations of equivariant parameters:  
\begin{equation}\label{eqn:specialization}
   \lambda_i=e^{\frac{2\pi\sqrt{-1}i}{5}}, \quad 0\leq i \leq 4.
\end{equation}
There are two differential equations, given precisely in Theorems \ref{thm:HAE_wrt_A2_partial} and \ref{thm:2nd_HAE} below. Theorem \ref{thm:HAE_wrt_A2_partial} is an analogue of the main result of \cite{lho-p} and \cite{lho-p2}. To the best of our knowledge, Theorem \ref{thm:2nd_HAE} does not have analogue in previous studies. Borrowing terminology from String Theory, we call these two differential equations {\em holomorphic anomaly equations} for $[\mathbb{C}^5/\mathbb{Z}_5]$.

Our approach to proving holomorphic anomaly equations is the same as that of \cite{lho-p2} and is based on the fact that genus $0$ Gromov-Witten theory of $[\mathbb{C}^5/\mathbb{Z}_5]$ yields a {\em semisimple} Frobenius structure on $H^*_\mathrm{T, Orb}([\mathbb{C}^5/\mathbb{Z}_5])$. Consequently, the {\em cohomological field theory} (in the sense of \cite{km}) associated to the Gromov-Witten theory of $[\mathbb{C}^5/\mathbb{Z}_5]$ is semisimple. The Givental-Teleman classification \cite{g3}, \cite{t} of semisimple cohomological field theories can then be applied to yield an explicit formula for $\mathcal{F}^{[\mathbb{C}^5/\mathbb{Z}_5]}_{g}$, which can be used to prove holomorphic anomaly equations.

The rest of the paper is organized as follows. In Section \ref{sec:mirror_thm}, we state the mirror theorem for $[\mathbb{C}^5/\mathbb{Z}_5]$ and study certain power series arising from the $I$-function. In Section \ref{subsec:frob}, we describe necessary ingredients of the Frobenius structure from Gromov-Witten theory of $[\mathbb{C}^5/\mathbb{Z}_5]$. In Section \ref{sec:lifts}, we study lifting to certain ring of functions of an important ingredient called the $R$-matrix. Section \ref{sec:HAE} contains the main results of this paper. In Section \ref{subsec:CohFT_formula}, we give the formula for Gromov-Witten potentials of $[\mathbb{C}^5/\mathbb{Z}_5]$ arising from Givental-Teleman classification of semisimple CohFTs. In Section \ref{subsec:HAE_proof}, we state the two holomorphic anomaly equations and use the formula in Section \ref{subsec:CohFT_formula} to prove them. 

\subsection{Acknowledgement}
D. G. is supported in part by a Special Graduate Assignment fellowship by Ohio State University's Department of Mathematics and H.-H. T. is supported in part by Simons foundation collaboration grant. 

\section{On mirror theorem}\label{sec:mirror_thm}
In this Section we discuss mirror theorem for Gromov-Witten theory of $[\mathbb{C}^5/\mathbb{Z}_5]$.

The $I$-function of $\left[\mathbb{C}^{5} / \mathbb{Z}_{5}\right]$ is defined\footnote{Here, $\langle \bullet \rangle$ is the fractional part of $\bullet$.} to be 
\begin{equation}\label{def:I-function}
I\left(x, z\right)=
        \sum_{k=0}^{\infty}\frac{x^k}{{z^k}k!}\prod_{\substack{b:0\leq b<\frac{k}{5} \\ \langle b \rangle=\langle\frac{k}{5}\rangle}}\left(1-(bz)^5\right)\phi_k.
\end{equation}
It is easy to see that $I$-function (\ref{def:I-function}) of $[\mathbb{C}^5/\mathbb{Z}_5]$ is of the form
\begin{equation}\label{eq:IfuncAsSumIk}
I\left(x, z\right)=\sum_{k=0}^{\infty}\frac{I_k(x)}{z^k}\phi_k.
\end{equation}

The small $J$-function for $[\mathbb{C}^5/\mathbb{Z}_5]$ is defined by 
\begin{equation*}
J\left(\Theta,z\right)=\phi_0+\frac{\Theta\phi_1}{z}+\sum_{i=0}^{n-1}\phi^i\left\langle\left\langle\frac{\phi_i}{z(z-\psi)}\right\rangle\right\rangle_{0,1}^{[\mathbb{C}^5/\mathbb{Z}_5]}.
\end{equation*}
The following is a consequence of the main result of \cite{ccit}.

\begin{prop} We have the following mirror identity,
\begin{equation}\label{eq:smallmirrorthm}
J\left(T(x),z\right)=I(x,z),
\end{equation}
with the mirror transformation\footnote{The gamma function is defined by $\Gamma(z)=\int_0^{\infty} t^{z-1} e^{-t} d t$ where $\Re(z)>0$. One of its fundamental properties is $\Gamma (z+1)=z \Gamma (z)$. We mainly use the gamma function in the expression of mirror transformation (\ref{eq:smallmirrorthm}) to give a compact formula.}
\begin{equation}\label{eq:mirrortransform}
T(x)=I_1(x)=\sum_{k\geq 0}\frac{(-1)^{5k}x^{5k+1}}{(5k+1)!}\left(\frac{\Gamma\left(k+\frac{1}{5}\right)}{\Gamma\left(\frac{1}{5}\right)}\right)^5.
\end{equation}
\end{prop}
Define the operator $$D:\mathbb{C}[[x]]\rightarrow x\mathbb{C}[[x]]$$ by
\begin{equation*}
Df(x)=x\frac{df(x)}{dx}.
\end{equation*}
Next, we consider the following series\footnote{Here, our $L$ differs from $L$ defined in \cite{lho} by a sign. Although the definitions of $C_1$ and $C_2$ look different from those in \cite{lho}, it is easy to check that these definitions match with those in \cite{lho}.} in $\mathbb{C}[[x]]$ arising from the $I$-function, which will be useful later:
\begin{equation}
\begin{aligned}
L=&x\left(1+\left(\frac{x}{5}\right)^5\right)^{-\frac{1}{5}},\\
C_1=&DI_1,\\
C_2=&D\left(\frac{DI_2}{C_1}\right),\\
C_3=&D\left(\frac{D\left(\frac{DI_3}{C_1}\right)}{C_2}\right).
\end{aligned}
\end{equation}
It is easy to verify that
\begin{equation}\label{eq:DLoverL}
\frac{DL}{L}=1-\frac{L^5}{5^5}.
\end{equation}
In \cite[Proposition 4]{lho}, the following identity is given:
\begin{equation}\label{eq:C12C22C3L5}
C_1^2C_2^2C_3=L^5.
\end{equation}

The following lemma is a direct result of the definition (\ref{eq:GW_Potential_phi_insertions}) of Gromov-Witten potential and the mirror map $\Theta=T(x)$.
\begin{lem}\label{lem:partialwrtmirrormap}
For $k\geq{1}$, we have
\begin{equation*}
\frac{\partial^k\mathcal{F}_{g, n}^{\left[\mathbb{C}^{5} / \mathbb{Z}_{5}\right]}\left(\phi_{c_{1}}, \ldots, \phi_{c_{n}}\right)}{\partial T^k}=\mathcal{F}_{g, n+k}^{\left[\mathbb{C}^{5} / \mathbb{Z}_{5}\right]}(\phi_{c_{1}}, \ldots, \phi_{c_{n}},\underbrace{\phi_1,\ldots,\phi_1}_{k-\text{many}}).
\end{equation*}
\end{lem}
We further define the following series:
\begin{equation}
\begin{aligned}
X_1=&\frac{DC_1}{C_1},\\
X_2=&\frac{DC_2}{C_2},\\
A_1=&\frac{1}{L}\left(\frac{DL}{L}-X_1\right),\\
A_2=&\frac{1}{L}\left(2\frac{DL}{L}-X_1-X_2\right),\\
B_i=&\frac{1}{5^i}(D+X_1)^{i-1}X_1\quad\text{for}\quad 1\leq{i}\leq{4}.
\end{aligned}
\end{equation}
In \cite[Section 3]{lho}, the following equations are given:
\begin{align}
B_4=&\left(1-\frac{L^5}{5^5}\right)\left(2 B_3-\frac{7}{5} B_2+\frac{2}{5} B_1-\frac{24}{625}\right), \\
DX_2=&-10\left(1-\frac{L^5}{5^5}\right)+10\left(1-\frac{L^5}{5^5}\right)X_1+5\left(1-\frac{L^5}{5^5}\right)X_2-2X_1^2-4DX_1-2X_1X_2-X_2^2 .\label{eq:DX2intermsofX1X2DX1}
\end{align}
Since there is a linear relation between $\{A_1,A_2\}$ and $\{X_1,X_2\}$ with coefficients from the ring $\mathbb{C}[L^{\pm 1}]$, we can rewrite these two equations in terms of $A_i$'s and their $D$ derivatives. For example, equation (\ref{eq:DX2intermsofX1X2DX1}) can be rewritten as
\begin{equation}\label{eq:DA2_intermsof_A1DA1A2}
DA_2=LA_1^2+LA_2^2-DA_1-15\left(1-\frac{L^5}{5^5}\right)\frac{L^5}{5^5}.
\end{equation}
Moreover, these linear relations show that the differential ring
\begin{equation*}
\mathbb{C}[L^{\pm 1}][A_1,A_2,DA_1,DA_2,D^2A_1,D^2A_2,\ldots]
\end{equation*}
is a quotient of the free polynomial ring
\begin{equation*}
\mathbb{F}\coloneqq\mathbb{C}[L^{\pm 1}][A_1,DA_1,D^2A_1,A_2].
\end{equation*}

\section{Frobenius Structures}\label{subsec:frob}
In this Section, we spell out details of the Frobenius structure on $H^{*}_{\mathrm{T,Orb}}\left(\left[\mathbb{C}^5/\mathbb{Z}_5\right]\right)$ defined using genus $0$ Gromov-Witten theory of $[\mathbb{C}^5/\mathbb{Z}_5]$. We refer to \cite{lp} for generalities of Frobenius structures.

Let $\gamma=\sum_{i=0}^{4}t_i\phi_i\in H^{*}_{\mathrm{T,Orb}}\left(\left[\mathbb{C}^5/\mathbb{Z}_5\right]\right)$. The full genus $0$ Gromov-Witten potential is defined to be
\begin{equation}\label{eqn:full_GW_potential}
    \mathcal{F}_0^{\left[\mathbb{C}^5 / \mathbb{Z}_5\right]}(t, \Theta)
    =\sum_{n=0}^{\infty} \sum_{d=0}^{\infty} \frac{1}{n ! d !}\int_{\left[\overline{M}_{0, n+d}^{\mathrm{orb}}\left(\left[\mathbb{C}^{5} / \mathbb{Z}_{5}\right], 0\right)\right]^{v i r}}  \prod_{i=1}^n \operatorname{ev}_i^*(\gamma) \prod_{i=n+1}^{n+d} \mathrm{ev}_i^*\left(\Theta \phi_1\right).
\end{equation}

In the basis $\{\phi_0,\phi_1,\phi_2, \phi_3,\phi_{4}\}$ and under the specialization (\ref{eqn:specialization}), the orbifold Poincar\'e pairing 
\begin{equation*}
    g(-,-): H^{*}_{\mathrm{T,Orb}}\left(\left[\mathbb{C}^5/\mathbb{Z}_5\right]\right)\times H^{*}_{\mathrm{T,Orb}}\left(\left[\mathbb{C}^5/\mathbb{Z}_5\right]\right)\to \mathbb{Q}(\lambda_0,\lambda_1, \lambda_2, \lambda_3,\lambda_{4})
\end{equation*}
has the matrix representation 
\begin{equation}
G=\frac{1}{5}
\begin{bmatrix}
1 & 0 & 0 & 0 & 0 \\
0 & 0 & 0 & 0 & 1 \\
0 & 0 & 0 & 1 & 0 \\
0 & 0 & 1 & 0 & 0 \\
0 & 1 & 0 & 0 & 0
\end{bmatrix}.
\end{equation}
The quantum product $\bullet_\gamma$ at $\gamma\in H^{*}_{\mathrm{T,Orb}}\left(\left[\mathbb{C}^5/\mathbb{Z}_5\right]\right)$ is a product structure on $H^{*}_{\mathrm{T,Orb}}\left(\left[\mathbb{C}^5/\mathbb{Z}_5\right]\right)$. It can be defined as follows:
\begin{equation*}
    g(\phi_i\bullet_\gamma \phi_j, \phi_k):=\frac{\partial^3}{\partial t_i\partial t_j\partial t_k}\mathcal{F}_0^{\left[\mathbb{C}^5 / \mathbb{Z}_5\right]}(t, \Theta).
\end{equation*}
In what follows, we focus on the quantum product $\bullet_{\gamma=0}$ at $\gamma=0\in H^{*}_{\mathrm{T,Orb}}\left(\left[\mathbb{C}^5/\mathbb{Z}_5\right]\right)$, which we denote by $\bullet$. Note that $\bullet$ still depends on $\Theta$.

It is proved in \cite[Section 5]{lho} that the quantum product at $0\in H^{*}_{\mathrm{T,Orb}}\left(\left[\mathbb{C}^5/\mathbb{Z}_5\right]\right)$ is semisimple with the idempotent basis $\{e_0,e_1, e_2, e_3 ,e_4\}$, that is,
\begin{equation*}
e_i\bullet e_j=\delta_{i,j}e_j.
\end{equation*}
The corresponding normalized idempotent basis $\{\widetilde{e}_0,\widetilde{e}_1, \widetilde{e}_2, \widetilde{e}_3,\widetilde{e}_4\}$ is given by
\begin{equation*}
\widetilde{e}_i=5e_i\quad\text{for}\quad 0\leq{i}\leq{4}.
\end{equation*}
In \cite[Section 5]{lho}, the transition matrix given by $\Psi_{ij}=g(\widetilde{e}_i,\phi_j)$ is calculated to be
\begin{equation*}
\Psi=\frac{1}{5}
\begin{bmatrix}
1 & \frac{L}{C_1} & \frac{L^2}{C_1 C_2} & \frac{C_1 C_2}{L^2} & \frac{C_1}{L} \\
1 & \zeta \frac{L}{C_1} & \zeta^2 \frac{L^2}{C_1 C_2} & \zeta^3 \frac{C_1 C_2}{L^2} & \zeta^4 \frac{C_1}{L} \\
1 & \zeta^2 \frac{L}{C_1} & \zeta^4 \frac{L^2}{C_1 C_2} & \zeta \frac{C_1 C_2}{L^2} & \zeta^3 \frac{C_1}{L} \\
1 & \zeta^3 \frac{L}{C_1} & \zeta \frac{L^2}{C_1 C_2} & \zeta^4 \frac{C_1 C_2}{L^2} & \zeta^2 \frac{C_1}{L} \\
1 & \zeta^4 \frac{L}{C_1} & \zeta^3 \frac{L^2}{C_1 C_2} & \zeta^2 \frac{C_1 C_2}{L^2} & \zeta \frac{C_1}{L}
\end{bmatrix}.
\end{equation*}

Let $\left\{u^{0},u^1, u^2, u^3,u^{4}\right\}$ be canonical coordinates associated to the idempotent basis $\left\{e_{0},e_1, e_2, e_3, e_{4}\right\}$ which satisfy
\begin{equation*}
u^{\alpha}\left(t_i=0,\Theta=0\right)=0.
\end{equation*}
By \cite[Lemma 6]{lho}, we have
\begin{equation}\label{canonicalcoorder2}
\frac{du^{\alpha}}{dx}=\zeta^{\alpha}L\frac{1}{x}
\end{equation}
at $t=0$, for $0\leq\alpha\leq{4}$.

The full genus $0$ Gromov-Witten potential (\ref{eqn:full_GW_potential}) satisfies the following property
\begin{equation*}
    \mathcal{F}_0^{\left[\mathbb{C}^5 / \mathbb{Z}_5\right]}(t, \Theta)=\mathcal{F}_0^{\left[\mathbb{C}^5 / \mathbb{Z}_5\right]}(t|_{t_1=0}, \Theta+t_1).
\end{equation*}
that is, $\mathcal{F}_0^{\left[\mathbb{C}^5 / \mathbb{Z}_5\right]}(t, \Theta)$ depends on $t_1$ and $\Theta$ through $\Theta+t_1$. In particular, the operator
\begin{equation}\label{annihilatoroperator}
\frac{\partial}{\partial{t_1}}-\frac{\partial}{\partial\Theta}
\end{equation}
annihilates $\mathcal{F}_0^{\left[\mathbb{C}^5 / \mathbb{Z}_5\right]}(t, \Theta)$.

Denote by 
\begin{equation*}
    R(z)=\text{Id}+\sum_{k\geq 1} R_k z^k\in \text{End}(H^*_{\mathrm{T,Orb}}([\mathbb{C}^5/\mathbb{Z}_5]))[[z]]
\end{equation*}
the $R$-matrix of the Frobenius structure associated to the ($\mathrm{T}$-equivariant) Gromov-Witten theory of $[\mathbb{C}^5/\mathbb{Z}_5]$ near the semisimple point $0$. The $R$-matrix plays a central role in the Givental-Teleman classification of semisimple cohomological field theories. By definition of $R$, the symplectic condition 
\begin{equation*}
    R(z)\cdot R(-z)^*=\text{Id}
\end{equation*}
holds. The following flatness equation
\begin{equation}\label{eqn:defn_of_R}
z(d\Psi^{-1})R+z\Psi^{-1}(dR)+\Psi^{-1}R (dU)-\Psi^{-1}(dU) R=0    
\end{equation}
also holds, see \cite[Section 4.6]{lp} and \cite[Proposition 1.1]{g1}. Here $d=\frac{d}{dt}$.

Since $\mathcal{F}_0^{\left[\mathbb{C}^5/ \mathbb{Z}_5\right]}(t, \Theta)$ depends on $t_1$ and $\Theta$ through $\Theta+t_1$, it follows that $\Psi$ and $R(z)$ also depend on $t_1$ and $\Theta$ through $\Theta+t_1$. So we have\footnote{An argument for this (written for a different target space) from the CohFT viewpoint can be found in \cite[Section 3.3]{pt}.} 
\begin{equation*}
\frac{\partial}{\partial t_1}\Psi=\frac{\partial}{\partial \Theta} \Psi, \quad
    \frac{\partial}{\partial t_1}R(z)=\frac{\partial}{\partial \Theta} R(z).
\end{equation*}
In equation (\ref{eqn:defn_of_R}), we set $t_{\neq 1}=0$ and only consider $\frac{d}{dt_1}$. It follows that (\ref{eqn:defn_of_R}) can be written as 
\begin{equation*}
   z(\frac{d}{d\Theta}\Psi^{-1})R+z\Psi^{-1}(\frac{d}{d\Theta}R)+\Psi^{-1}R (\frac{d}{d\Theta}U)-\Psi^{-1}(\frac{d}{d\Theta}U) R=0. 
\end{equation*}
Since $$\frac{d}{d\Theta}=\frac{dx}{d\Theta}\frac{d}{dx},$$
after cancelling $\frac{dx}{d\Theta}$ and multiplying by $x$, we rewrite the above equation as 
\begin{equation*}
   z(x\frac{d}{dx}\Psi^{-1})R+z\Psi^{-1}(x\frac{d}{dx}R)+\Psi^{-1}R (x\frac{d}{dx}U)-\Psi^{-1}(x\frac{d}{dx}U) R=0. 
\end{equation*}
By equating coefficients of $z^k$, we see that 
\begin{equation}\label{flatness1}
\Psi\left(D\Psi^{-1}\right)R_{k-1}+DR_{k-1}+R_k\left(DU\right)-\left(DU\right)R_k=0
\end{equation}
or equivalently,
\begin{equation}\label{flatness2}
D\left(\Psi^{-1}R_{k-1}\right)+\left(\Psi^{-1}R_k\right)DU-\Psi^{-1}\left(DU\right)\Psi\left(\Psi^{-1}R_k\right)=0   
\end{equation}
where $D=x\frac{d}{dx}$ as before.  

Now set $t_1=0$. By equation (\ref{canonicalcoorder2}), we have 
\begin{equation}\label{DU}
DU=\diag(L,L\zeta,L\zeta^2,L\zeta^3,L\zeta^4).
\end{equation}

Let $P_{i,j}^k$ denote the $(i,j)$-entry of the matrix defined by $P_k=\Psi^{-1}R_k$ after being restricted to the semisimple point $0\in H^*_{\mathrm{T,Orb}}\left([\mathbb{C}^5/\mathbb{Z}_5]\right)$. Set \begin{equation*}
\widetilde{P}_{i,j}^k=\frac{L^i}{K_i}P_{i,j}^k\zeta^{(k+i)j}
\end{equation*}
where $0\leq i,j \leq{4}$ and $k\geq 0$. Then, the flatness equation (\ref{flatness2}) reads as
\begin{equation}\label{eq:modifiedflatness}
\begin{aligned}
\widetilde{P}_{4,j}^k=&\widetilde{P}_{0,j}^k+\frac{1}{L}D\widetilde{P}_{0,j}^{k-1},\\
\widetilde{P}_{3,j}^k=&\widetilde{P}_{4,j}^k+\frac{1}{L}D\widetilde{P}_{4,j}^{k-1}+A_1\widetilde{P}_{4,j}^{k-1},\\
\widetilde{P}_{2,j}^k=&\widetilde{P}_{3,j}^k+\frac{1}{L}D\widetilde{P}_{3,j}^{k-1}+A_2\widetilde{P}_{3,j}^{k-1},\\
\widetilde{P}_{1,j}^k=&\widetilde{P}_{2,j}^k+\frac{1}{L}D\widetilde{P}_{2,j}^{k-1}-A_2\widetilde{P}_{2,j}^{k-1},\\
\widetilde{P}_{0,j}^k=&\widetilde{P}_{1,j}^k+\frac{1}{L}D\widetilde{P}_{1,j}^{k-1}-A_1\widetilde{P}_{1,j}^{k-1}.
\end{aligned}
\end{equation}
We call equation (\ref{eq:modifiedflatness}) the \textit{modified flatness equations}.

In \cite{zz}, properties of some hypergeometric series associated with mirror symmetry are analyzed, and certain formal power series are associated to these hypergeometric series. They obtained equations like (\ref{eq:C12C22C3L5}), and showed that some parts of the asymptotic expansions of those hypergeometric series lie in polynomial rings in those formal power series. By the methodology of \cite{zz}, we obtain the following result\footnote{The same result for the total space $K\mathbb{P}^4$ of the canonical bundle of $\mathbb{P}^4$ can be inferred from \cite[Proposition 14]{lho}. Hence, Lemma \ref{lem:P0jk_is_in_CLplusminus} can also be deducted from the main result of \cite{lho}.}.
\begin{lem}\label{lem:P0jk_is_in_CLplusminus}
We have $\widetilde{P}_{0,j}^k\in\mathbb{C}[L^{\pm 1}]$ for all $0\leq{j}\leq{4}$ and $k\geq{0}$.
\end{lem}

\section{Lift of modified R-matrix}\label{sec:lifts}
\subsection{Canonical Lift}
The functions $\widetilde{P}_{i,j}^k\in\mathbb{C}[[x]]$ in modified flatness equations have canonical lifts to the the ring  $$\mathbb{F}=\mathbb{C}[L^{\pm 1}][A_1,DA_1,D^2A_1,A_2]$$ 
via equation (\ref{eq:DLoverL}), equation (\ref{eq:DA2_intermsof_A1DA1A2}) and the first four rows of the flatness equations (\ref{eq:modifiedflatness}) in the descending order. More precisely, we start with Lemma \ref{lem:P0jk_is_in_CLplusminus}, that is
\begin{equation*}
\widetilde{P}_{0,j}^k\in\mathbb{C}[L^{\pm 1}]\subset\mathbb{F}.
\end{equation*}
Then, by equation (\ref{eq:DLoverL}), we obtain
\begin{equation*}
\widetilde{P}_{4,j}^k=\widetilde{P}_{0,j}^k+\frac{1}{L}D\widetilde{P}_{0,j}^{k-1}\in\mathbb{C}[L^{\pm 1}]\subset\mathbb{F}.
\end{equation*}
By proceeding in a similar manner, we see that
\begin{equation*}
\begin{aligned}
\widetilde{P}_{3,j}^k&=\widetilde{P}_{4,j}^k+\frac{1}{L}D\widetilde{P}_{4,j}^{k-1}+A_1\widetilde{P}_{4,j}^{k-1}\in\mathbb{C}[L^{\pm 1}][A_1]\subset\mathbb{F},\\
\widetilde{P}_{2,j}^k&=\widetilde{P}_{3,j}^k+\frac{1}{L}D\widetilde{P}_{3,j}^{k-1}+A_2\widetilde{P}_{3,j}^{k-1}\in\mathbb{C}[L^{\pm 1}][A_1,DA_1,A_2]\subset\mathbb{F}.
\end{aligned}
\end{equation*}
Lastly, using equation (\ref{eq:DA2_intermsof_A1DA1A2}) we get
\begin{equation*}
\widetilde{P}_{1,j}^k=\widetilde{P}_{2,j}^k+\frac{1}{L}D\widetilde{P}_{2,j}^{k-1}-A_2\widetilde{P}_{2,j}^{k-1}\in\mathbb{C}[L^{\pm 1}][A_1,DA_1,D^2A_1,A_2]=\mathbb{F}.
\end{equation*}
This procedure gives us a canonical lift of $\widetilde{P}_{i,j}^k\in\mathbb{C}[[x]]$ to the free polynomial ring $\mathbb{F}$, which we denote again as $\widetilde{P}_{i,j}^k$. We state this result in the following way.

\begin{lem}\label{lem:lifting_lemma}
We have $\widetilde{P}_{i,j}^k\in\mathbb{F}$ for all $0\leq{i,j}\leq{4}$ and $k\geq{0}$.
\end{lem}

\subsection{Preparations}
In this subsection, we use the lift $\widetilde{P}_{i,j}^k\in\mathbb{F}$ and prove two lemmas which will be used for the proof of holomorphic anomaly equations.
\begin{lem}\label{lem:derivativeof_Pijk_wrtA2}
The following identity holds
\begin{equation*}
\frac{\partial\widetilde{P}_{i,j}^k}{\partial{A_2}}=\delta_{i,2}\widetilde{P}_{{3},j}^{k-1}.
\end{equation*}
\end{lem}
\begin{proof}
It is clear that the degrees of $\widetilde{P}_{0,j}^k$, $\widetilde{P}_{4,j}^k$ and $\widetilde{P}_{3,j}^k$ in $A_2$ are all zero. Hence, we get
\begin{equation}
\frac{\partial\widetilde{P}_{0,j}^k}{\partial{A_2}}=\frac{\partial\widetilde{P}_{4,j}^k}{\partial{A_2}}=\frac{\partial\widetilde{P}_{3,j}^k}{\partial{A_2}}=0.
\end{equation}
The place where we see $A_2$ for the first time are the next two equations in (\ref{eq:modifiedflatness}),
\begin{align}
\widetilde{P}_{2,j}^k=&\widetilde{P}_{3,j}^k+\frac{1}{L}D\widetilde{P}_{3,j}^{k-1}+A_2\widetilde{P}_{3,j}^{k-1},\label{eq:A2middleequations1}\\
\widetilde{P}_{1,j}^k=&\widetilde{P}_{2,j}^k+\frac{1}{L}D\widetilde{P}_{2,j}^{k-1}-A_2\widetilde{P}_{2,j}^{k-1},\label{eq:A2middleequations2}
\end{align}
From the first equation (\ref{eq:A2middleequations1}), we see that
\begin{equation*}\label{eq:A2derivative_of_P2jk}
\frac{\partial\widetilde{P}_{2,j}^k}{\partial{A_2}}=\widetilde{P}_{3,j}^{k-1}.
\end{equation*}
Note that equation (\ref{eq:DA2_intermsof_A1DA1A2}) gives
\begin{equation*}\label{eq:A2derivative_of_DA2}
\frac{\partial \left(DA_2\right)}{\partial A_2}=2LA_2.
\end{equation*}
Now, we compute the last derivative. By flatness equations (\ref{eq:modifiedflatness}) and equation (\ref{eq:A2derivative_of_DA2}), we obtain
\begin{equation}
\begin{aligned}
\frac{\partial \widetilde{P}_{1,j}^k}{\partial A_2}=&\frac{\partial \widetilde{P}_{2,j}^k}{\partial A_2}+\frac{1}{L}\frac{\partial\left(D\widetilde{P}_{2,j}^{k-1}\right)}{\partial A_2}-\widetilde{P}_{2,j}^{k-1}-A_2\frac{\partial \widetilde{P}_{2,j}^{k-1}}{\partial A_2}\\
=&\frac{\partial \widetilde{P}_{2,j}^k}{\partial A_2}+\frac{1}{L}\left(2LA_2\widetilde{P}_{3,j}^{k-2}+D\widetilde{P}_{3,j}^{k-2}\right)-\widetilde{P}_{2,j}^{k-1}-A_2\frac{\partial \widetilde{P}_{2,j}^{k-1}}{\partial A_2}.
\end{aligned}
\end{equation}
Then, by equation (\ref{eq:A2derivative_of_P2jk}) and again by flatness equations (\ref{eq:modifiedflatness}), we get
\begin{equation}
\begin{aligned}
\frac{\partial \widetilde{P}_{1,j}^k}{\partial A_2}
=&\widetilde{P}_{3,j}^{k-1}+2A_2\widetilde{P}_{3,j}^{k-2}+\frac{1}{L}D\widetilde{P}_{3,j}^{k-2}-\widetilde{P}_{2,j}^{k-1}-A_2\widetilde{P}_{3,j}^{k-2}\\
=&\widetilde{P}_{3,j}^{k-1}+A_2\widetilde{P}_{3,j}^{k-2}+\frac{1}{L}D\widetilde{P}_{3,j}^{k-2}-\widetilde{P}_{2,j}^{k-1}=0.
\end{aligned}
\end{equation}
This completes the proof.
\end{proof}
 
\begin{lem}\label{lem:derivativeof_Pijk_wrtDA1}
The following identity holds
\begin{equation*}
\frac{\partial\widetilde{P}_{i,j}^k}{\partial{(D^2A_1)}}=\delta_{i,1}\frac{1}{L^2}\widetilde{P}_{{4},j}^{k-3}.
\end{equation*}
\end{lem}
\begin{proof}
It is clear that the only $\widetilde{P}_{i,j}^k\in\mathbb{F}$ that has non-zero degree in $D^2A_1$ is $\widetilde{P}_{1,j}^k$, and the degree of $\widetilde{P}_{1,j}^k$ in $D^2A_1$ is $1$. So, we obtain
\begin{equation}
\frac{\partial\widetilde{P}_{0,j}^k}{\partial{(D^2A_1)}}=\frac{\partial\widetilde{P}_{4,j}^k}{\partial{(D^2A_1)}}=\frac{\partial\widetilde{P}_{3,j}^k}{\partial{(D^2A_1)}}=\frac{\partial\widetilde{P}_{2,j}^k}{\partial{(D^2A_1)}}=0.
\end{equation}
The coefficient of $D^2A_1$ in $\widetilde{P}_{1,j}^k$ descends from the coefficient of $A_1$ in $\widetilde{P}_{3,j}^{k-2}$, which is $\widetilde{P}_{4,j}^{k-3}$. Keeping track of this term in the procedure of canonical lifting, we see that the coefficient of $D^2A_1$ in $\widetilde{P}_{1,j}^k$ is
\begin{equation*}
\frac{1}{L^2}\widetilde{P}_{{4},j}^{k-3}.
\end{equation*}
This completes the proof.
\end{proof}


\section{Holomorphic anomaly equations}\label{sec:HAE}
\subsection{Formula for potentials}\label{subsec:CohFT_formula}
By general considerations, Gromov-Witten theory of $[\mathbb{C}^5/\mathbb{Z}_5]$ has the structure of a cohomological field theory (CohFT). We refer to \cite{km} and \cite{Picm} for discussions on CohFTs. 

\subsubsection{Graphs}
We describe the graphs needed in the formula for Gromov-Witten potentials.

Recall that a \textit{stable graph} $\Gamma$ is a tuple
\begin{equation*}
\Gamma=\left(\mathrm{V}_{\Gamma}, \mathrm{H}_{\Gamma}, \mathrm{L}_{\Gamma}, \mathrm{g}: \mathrm{V}_{\Gamma} \rightarrow \mathbb{Z}_{\geq 0}, \nu: \mathrm{H}_{\Gamma}\cup\mathrm{L}_{\Gamma} \rightarrow \mathrm{V}_{\Gamma}, \iota: \mathrm{H}_{\Gamma} \rightarrow \mathrm{H}_{\Gamma}, \ell:\mathrm{L}_{\Gamma}\rightarrow\{1,\ldots,n\} \right)
\end{equation*}
satisfying:
\begin{enumerate}
\item $\mathrm{V}_{\Gamma}$ is the vertex set with a genus assignment $\mathrm{g}:\mathrm{V}_{\Gamma}\rightarrow\mathbb{Z}_{\geq 0}$,
    
\item $\mathrm{H}_{\Gamma}$ is the half-edge set equipped with an involution $\iota: \mathrm{H}_{\Gamma} \rightarrow \mathrm{H}_{\Gamma}$,
    
\item $\mathrm{E}_{\Gamma}$ is the edge set defined by the orbits of $\iota: \mathrm{H}_{\Gamma} \rightarrow \mathrm{H}_{\Gamma}$ 
in $\mathrm{H}_{\Gamma}$ (self-edges are allowed at the vertices) and the tuple $\left(\mathrm{V}_{\Gamma},\mathrm{E}_{\Gamma}\right)$ defines a connected graph,
    
\item $\mathrm{L}_{\Gamma}$ is the set of legs, the subset of $\mathrm{H}_{\Gamma}$ fixed by the involution $\iota: \mathrm{H}_{\Gamma} \rightarrow \mathrm{H}_{\Gamma}$ and the map $\ell:\mathrm{L}_{\Gamma}\rightarrow\{1,\ldots,m\}$ is an isomorphism labeling legs,

\item The map $\nu: \mathrm{H}_{\Gamma}\cup\mathrm{L}_{\Gamma} \rightarrow \mathrm{V}_{\Gamma}$ is a vertex assignment,
    
\item For each vertex $v$, let $\mathrm{n}(\mathfrak{v})=\mathrm{l}(\mathfrak{v})+\mathrm{h}(\mathfrak{v})$ be the valence of the vertex (where $\mathrm{l}(\mathfrak{v})$ and $\mathrm{h}(\mathfrak{v})$ are the number of legs and the number of edges attached to the vertex $\mathfrak{v}$ respectively). Then, the following stability condition holds:
\begin{equation*}
2\mathrm{g}(\mathfrak{v})-2+\mathrm{n}(\mathfrak{v})>0.
\end{equation*}
\end{enumerate}

The \textit{genus} of a stable graph $\Gamma$ is defined by
\begin{equation*}
\mathrm{g}(\Gamma)=h^1(\Gamma)+\sum_{\mathfrak{v}\in\mathrm{V}_{\Gamma}}\mathrm{g}(\mathfrak{v}).
\end{equation*}
Let $\mathrm{G}_{g,m}$ be the isomorphism classes of stable graphs of genus $g$ with $n$ legs. A \textit{decorated stable graph} $$\Gamma\in\mathrm{G}_{g,n}^{\text{Dec}}(5)$$ of order $5$ is a stable graph $\Gamma\in\mathrm{G}_{g,n}$ with an extra assignment $\mathrm{p}: \mathrm{V}_{\Gamma}\rightarrow \{0,1, 2, 3,4\}$ to each vertex $\mathfrak{v}\in\mathrm{V}_{\Gamma}$. For a decorated stable graph $\Gamma\in\mathrm{G}_{g,n}^{\text{Dec}}(5)$ we denote its underlying stable graph by $$\Gamma^{\mathrm{St}}\in\mathrm{G}_{g,n}$$ after forgetting the decoration.

\subsubsection{Formula for $\mathcal{F}_g$}
By the results stated in Section \ref{subsec:frob}, the CohFT of Gromov-Witten theory of $[\mathbb{C}^5/\mathbb{Z}_5]$ is semisimple. By Givental-Teleman classification of semisimple CohFTs (see e.g. \cite{Picm} for a survey), we can write Gromov-Witten potential as a sum over decorated stable graphs,
\begin{equation}\label{eqn:formula_Fg}
\mathcal{F}_{g, n}^{\left[\mathbb{C}^{5} / \mathbb{Z}_{5}\right]}\left(\phi_{c_{1}}, \ldots, \phi_{c_{n}}\right)=\sum_{\Gamma\in\mathrm{G}_{g,m}^{\text{Dec}}(5)}\operatorname{Cont}_{\Gamma}\left(\phi_{c_{1}}, \ldots, \phi_{c_{n}}\right).
\end{equation}
Details about how this formula works in general can be found in e.g. \cite{ppz} and \cite{lho-p2}.

In order to state the contributions of graphs to $\mathcal{F}_{g, n}^{\left[\mathbb{C}^{5} / \mathbb{Z}_{5}\right]}\left(\phi_{c_{1}}, \ldots, \phi_{c_{n}}\right)$, we need to introduce the following series in $\mathbb{C}[[x]]$:
\begin{equation*}
K_0,\quad K_1=C_1,\quad K_2=C_1C_2,\quad K_3=C_1C_2C_3,\quad\text{and}\quad K_4=C_1C_2^2C_3,
\end{equation*}
and the following involution
$$\mathrm{Inv}:\{0,1,2,3,4\}\rightarrow\{0,1,2,3,4\},$$ with $\mathrm{Inv}(0)=0$ and $\mathrm{Inv}(i)=5-i$ for $1\leq i \leq 4$.
\begin{prop}\label{prop:contributions}
The contribution associated to a decorated stable graph $\Gamma\in\mathrm{G}_{g,n}^{\text{Dec}}(5)$ is
\begin{equation*}
\operatorname{Cont}_{\Gamma}\left(\phi_{c_{1}}, \ldots, \phi_{c_{n}}\right)=\frac{1}{|\mathrm{Aut}(\Gamma^{\mathrm{St}})|} \sum_{\mathrm{A} \in \mathbb{Z}_{\geq 0}^{\mathrm{F}(\Gamma)}} \prod_{\mathfrak{v} \in \mathrm{V}_{\Gamma}} \mathrm{Cont}_{\Gamma}^{\mathrm{A}}(\mathfrak{v}) \prod_{\mathfrak{e}\in \mathrm{E}_{\Gamma}} \mathrm{Cont}_{\Gamma}^{\mathrm{A}}(\mathfrak{e}) \prod_{\mathfrak{l} \in \mathrm{L}_{\Gamma}} \mathrm{Cont}_{\Gamma}^{\mathrm{A}}(\mathfrak{l})
\end{equation*}
where $\mathrm{F}(\Gamma)=\left\vert\mathrm{H}_{\Gamma}\right\vert$. Here, $\mathrm{Cont}_{\Gamma}^{\mathrm{A}}(\mathfrak{v})$, $\mathrm{Cont}_{\Gamma}^{\mathrm{A}}(\mathfrak{e})$, and $\mathrm{Cont}_{\Gamma}^{\mathrm{A}}(\mathfrak{l})$ are the {\em vertex}, {\em edge} and {\em leg} contributions with flag $\mathrm{A}-$values\footnote{Notation: The values ${b_{\mathfrak{v}1}},\ldots,{b_{\mathfrak{v}\mathrm{h}(\mathfrak{v})}}$ and $b_{\mathfrak{e}1},b_{\mathfrak{e}2}$ are the entries of $(a_1,\ldots,a_m,b_{m+1},\ldots,b_{\left\vert\mathrm{H}_{\Gamma}\right\vert})$ corresponding to $\mathrm{Cont}_{\Gamma}^{\mathrm{A}}(\mathfrak{v})$ and $\mathrm{Cont}_{\Gamma}^{\mathrm{A}}(\mathfrak{e})$ respectively.} $(a_1,\ldots,a_m,b_{m+1},\ldots,b_{\left\vert\mathrm{H}_{\Gamma}\right\vert})$ respectively, and they are given by
\begin{equation*}
\begin{aligned}
    \mathrm{Cont}_{\Gamma}^{\mathrm{A}}(\mathfrak{v})
    =&\sum_{k \geq 0} \frac{g({e}_{\mathrm{p}(\mathfrak{v})},{e}_{\mathrm{p}(\mathfrak{v})})^{-\frac{2\mathrm{g}(\mathfrak{v})-2+\mathrm{n}(\mathfrak{v})+k}{2}}}{k !}\\
    &\times\int_{\overline{M}_{\mathrm{g}(\mathfrak{v}),\mathrm{n}(\mathfrak{v})+k}}\psi_1^{a_{\mathfrak{v}1}}\cdots\psi_{\mathrm{l}(v)}^{a_{\mathfrak{v}\mathrm{l}(\mathfrak{v})}}\psi_{\mathrm{l}(\mathfrak{v})+1}^{b_{\mathfrak{v}1}}\cdots\psi_{\mathrm{n}(\mathfrak{v})}^{b_{\mathfrak{v}\mathrm{h}(\mathfrak{v})}}t_{\mathrm{p}(\mathfrak{v})}(\psi_{\mathrm{n}(\mathfrak{v})+1})\cdots t_{\mathrm{p}(\mathfrak{v})}(\psi_{\mathrm{n}(\mathfrak{v})+k}),\\
    \mathrm{Cont}_{\Gamma}^{\mathrm{A}}(\mathfrak{e})
    =&\frac{(-1)^{b_{\mathfrak{e}1}+b_{\mathfrak{e}2}}}{5} \sum_{m=0}^{b_{\mathfrak{e}2}}(-1)^{m} \sum_{r=0}^{4}\frac{\widetilde{P}_{\mathrm{Inv}(r),\mathrm{p}(\mathfrak{v}_1)}^{b_{\mathfrak{e}1}+m+1}\widetilde{P}_{r,\mathrm{p}(\mathfrak{v}_2)}^{b_{\mathfrak{e}2}-m}}{\zeta^{(b_{\mathfrak{e}1}+m+1+\mathrm{Inv}(r))\mathrm{p}(\mathfrak{v}_1)}\zeta^{(b_{\mathfrak{e}2}-m+r)\mathrm{p}(\mathfrak{v}_2)}},\\
    \mathrm{Cont}_{\Gamma}^{\mathrm{A}}(\mathfrak{l})
    =&\frac{(-1)^{a_{\ell(\mathfrak{l})}}}{5}\frac{K_{\mathrm{Inv}(c_{\ell(\mathfrak{l})})}}{L^{\mathrm{Inv}(c_{\ell(\mathfrak{l})})}}
    \frac{\widetilde{P}_{\mathrm{Inv}(c_{\ell(\mathfrak{l})}),\mathrm{p}(\nu(\mathfrak{l}))}^{{a_{\ell(\mathfrak{l})}}}}{     \zeta^{({a_{\ell(\mathfrak{l})}}+{\mathrm{Inv}(c_{\ell(\mathfrak{l})})})\mathrm{p}(\nu(\mathfrak{l}))}},
\end{aligned}
\end{equation*}
where
\begin{equation*}
t_{\mathrm{p}(\mathfrak{v})}(z)=\sum_{j\geq{2}}\mathrm{T}_{\mathrm{p}(\mathfrak{v})j}z^j\quad\text{with}\quad \mathrm{T}_{\mathrm{p}(\mathfrak{v})j}=\frac{(-1)^j}{n}\widetilde{P}_{0,\mathrm{p}(\mathfrak{v})}^{j-1}\zeta^{-(j-1)\mathrm{p}(\mathfrak{v})}.
\end{equation*}
\end{prop}
We should emphasize that Proposition \ref{prop:contributions} holds in $\mathbb{C}[[x]]$. Using Proposition \ref{prop:contributions} and lifting procedure in Section \ref{sec:lifts}, we obtain the following lift of Gromov-Witten potential to certain polynomial rings.

\begin{thm}[Finite generation property]\label{thm:vertex_edge_leg_Rings}
Let $\mathrm{Cont}_{\Gamma}^{\mathrm{A}}(\mathfrak{v})$, $\mathrm{Cont}_{\Gamma}^{\mathrm{A}}(\mathfrak{e})$, and $\mathrm{Cont}_{\Gamma}^{\mathrm{A}}(\mathfrak{l})$ be as in Proposition \ref{prop:contributions}. We have $\mathrm{Cont}_{\Gamma}^{\mathrm{A}}(\mathfrak{v})\in\mathbb{C}[L^{\pm 1}]$, $\mathrm{Cont}_{\Gamma}^{\mathrm{A}}(\mathfrak{e})\in\mathbb{F}$, and $\mathrm{Cont}_{\Gamma}^{\mathrm{A}}(\mathfrak{l})\in\mathbb{F}[C_1,C_2,C_3]$. Hence, we have
\begin{equation*}
\mathcal{F}_{g, n}^{\left[\mathbb{C}^{5} / \mathbb{Z}_{5}\right]}\left(\phi_{c_{1}}, \ldots, \phi_{c_{n}}\right)\in\mathbb{F}[C_1,C_2,C_3]
\end{equation*}
and when there is no insertions, we have
\begin{equation*}
\mathcal{F}_{g}^{\left[\mathbb{C}^{5} / \mathbb{Z}_{5}\right]}\in\mathbb{F}
\end{equation*}
where $\mathbb{F}=\mathbb{C}[L^{\pm 1}][A_1,DA_1,D^2A_1,A_2]$ as before.
\end{thm}

\begin{proof}
By Lemma \ref{lem:P0jk_is_in_CLplusminus}, we have $\mathrm{Cont}_{\Gamma}^{\mathrm{A}}(\mathfrak{v})\in\mathbb{C}[L^{\pm 1}]$ since its expression involves only $\widetilde{P}_{0,j}^k$'s. By Lemma \ref{lem:lifting_lemma}, and definitions of $K_i$'s, we see that $\mathrm{Cont}_{\Gamma}^{\mathrm{A}}(\mathfrak{e})\in\mathbb{F}$ and $\mathrm{Cont}_{\Gamma}^{\mathrm{A}}(\mathfrak{l})\in\mathbb{F}[C_1, C_2, C_3]$. Hence, results for Gromov-Witten potentials follow.
\end{proof}

Depending on the insertions, we can give a better description of the polynomial ring that contains  Gromov-Witten potentials. For example, by Proposition \ref{prop:contributions} we have
\begin{equation}
\mathcal{F}_{g, n}^{\left[\mathbb{C}^{5} / \mathbb{Z}_{5}\right]}\left(\phi_{c_{1}}, \ldots, \phi_{c_{1}}\right)\in\mathbb{F}[C_1^{-1}]=\mathbb{C}[L^{\pm 1}][A_1,DA_1,D^2A_1,A_2,C_1^{-1}]
\end{equation}
and the degree of $C_1^{-1}$ in $\mathcal{F}_{g, n}^{\left[\mathbb{C}^{5} / \mathbb{Z}_{5}\right]}\left(\phi_{c_{1}}, \ldots, \phi_{c_{1}}\right)$ is $n$. Then, we obtain the following result by Lemma \ref{lem:partialwrtmirrormap}.

\begin{cor}
For all, $k\geq{1}$, we have
\begin{equation*}
\frac{\partial^k\mathcal{F}_{g}^{\left[\mathbb{C}^{5} / \mathbb{Z}_{5}\right]}}{\partial T^k}\in\mathbb{F}[C_1^{-1}]=\mathbb{C}[L^{\pm 1}][A_1,DA_1,D^2A_1,A_2,C_1^{-1}]
\end{equation*}
and the degree of $C_1^{-1}$ in $\frac{\partial^k\mathcal{F}_{g}^{\left[\mathbb{C}^{5} / \mathbb{Z}_{5}\right]}}{\partial T^k}$ is $k$.
\end{cor}

\subsection{Proof of holomorphic anomaly equations}\label{subsec:HAE_proof}

\begin{lem}\label{lem:A2_derivative_of_edge}
We have
\begin{equation*} 
    \frac{\partial}{\partial A_{2}}\mathrm{Cont}_{\Gamma}^{\mathrm{A}}(\mathfrak{e})
    =\frac{(-1)^{b_{\mathfrak{e}1}+b_{\mathfrak{e}2}}}{5} \frac{\widetilde{P}_{3,\mathrm{p}(\mathfrak{v}_1)}^{b_{\mathfrak{e}1}}\widetilde{P}_{3,\mathrm{p}(\mathfrak{v}_2)}^{b_{\mathfrak{e}2}}}{\zeta^{(b_{\mathfrak{e}1}+3)\mathrm{p}(\mathfrak{v}_1)}\zeta^{(b_{\mathfrak{e}2}+3)\mathrm{p}(\mathfrak{v}_2)}}.
\end{equation*}
\end{lem}
\begin{proof}
By Proposition \ref{prop:contributions} and Lemma \ref{lem:derivativeof_Pijk_wrtA2}, we obtain
\begin{equation*}
\begin{aligned}
    \frac{\partial}{\partial A_2}\mathrm{Cont}_{\Gamma}^{\mathrm{A}}(\mathfrak{e})
    =&\frac{(-1)^{b_{\mathfrak{e}1}+b_{\mathfrak{e}2}}}{5} \sum_{m=0}^{b_{\mathfrak{e}2}}(-1)^{m} \sum_{r=0}^{4}\frac{\frac{\partial}{\partial A_2}\left(\widetilde{P}_{\mathrm{Inv}(r),\mathrm{p}(\mathfrak{v}_1)}^{b_{\mathfrak{e}1}+m+1}\widetilde{P}_{r,\mathrm{p}(\mathfrak{v}_2)}^{b_{\mathfrak{e}2}-m}\right)}{\zeta^{(b_{\mathfrak{e}1}+m+1+\mathrm{Inv}(r))\mathrm{p}(\mathfrak{v}_1)}\zeta^{(b_{\mathfrak{e}2}-m+r)\mathrm{p}(\mathfrak{v}_2)}}\\
    =&\frac{(-1)^{b_{\mathfrak{e}1}+b_{\mathfrak{e}2}}}{5} \sum_{m=0}^{b_{\mathfrak{e}2}}(-1)^{m} \frac{\widetilde{P}_{3,\mathrm{p}(\mathfrak{v}_1)}^{b_{\mathfrak{e}1}+m}\widetilde{P}_{3,\mathrm{p}(\mathfrak{v}_2)}^{b_{\mathfrak{e}2}-m}}{\zeta^{(b_{\mathfrak{e}1}+m+3)\mathrm{p}(\mathfrak{v}_1)}\zeta^{(b_{\mathfrak{e}2}-m+3)\mathrm{p}(\mathfrak{v}_2)}}\\
    &+\frac{(-1)^{b_{\mathfrak{e}1}+b_{\mathfrak{e}2}}}{5} \sum_{m=0}^{b_{\mathfrak{e}2}}(-1)^{m}\frac{\widetilde{P}_{3,\mathrm{p}(\mathfrak{v}_1)}^{b_{\mathfrak{e}1}+m+1}\widetilde{P}_{3,\mathrm{p}(\mathfrak{v}_2)}^{b_{\mathfrak{e}2}-m-1}}{\zeta^{(b_{\mathfrak{e}1}+m+4)\mathrm{p}(\mathfrak{v}_1)}\zeta^{(b_{\mathfrak{e}2}-m+2)\mathrm{p}(\mathfrak{v}_2)}}.
\end{aligned}
\end{equation*}
Since $\widetilde{P}_{i,j}^k$ is defined to be $0$ for $k<0$, the second summation ends actually at $m=b_{\mathfrak{e}2}-1$. Then, by shifting the second summmation by $1$, and cancelling out terms in total expression, we get
\begin{equation*}
\begin{aligned}
    \frac{\partial}{\partial A_2}\mathrm{Cont}_{\Gamma}^{\mathrm{A}}(\mathfrak{e})
    =&\frac{(-1)^{b_{\mathfrak{e}1}+b_{\mathfrak{e}2}}}{5} \sum_{m=0}^{b_{\mathfrak{e}2}}(-1)^{m} \frac{\widetilde{P}_{3,\mathrm{p}(\mathfrak{v}_1)}^{b_{\mathfrak{e}1}+m}\widetilde{P}_{3,\mathrm{p}(\mathfrak{v}_2)}^{b_{\mathfrak{e}2}-m}}{\zeta^{(b_{\mathfrak{e}1}+m+3)\mathrm{p}(\mathfrak{v}_1)}\zeta^{(b_{\mathfrak{e}2}-m+3)\mathrm{p}(\mathfrak{v}_2)}}\\
    &+\frac{(-1)^{b_{\mathfrak{e}1}+b_{\mathfrak{e}2}}}{5} \sum_{m=1}^{b_{\mathfrak{e}2}}(-1)^{m-1} \frac{\widetilde{P}_{3,\mathrm{p}(\mathfrak{v}_1)}^{b_{\mathfrak{e}1}+m}\widetilde{P}_{3,\mathrm{p}(\mathfrak{v}_2)}^{b_{\mathfrak{e}2}-m}}{\zeta^{(b_{\mathfrak{e}1}+m+3)\mathrm{p}(\mathfrak{v}_1)}\zeta^{(b_{\mathfrak{e}2}-m+3)\mathrm{p}(\mathfrak{v}_2)}}\\
    =&\frac{(-1)^{b_{\mathfrak{e}1}+b_{\mathfrak{e}2}}}{5} \frac{\widetilde{P}_{3,\mathrm{p}(\mathfrak{v}_1)}^{b_{\mathfrak{e}1}}\widetilde{P}_{3,\mathrm{p}(\mathfrak{v}_2)}^{b_{\mathfrak{e}2}}}{\zeta^{(b_{\mathfrak{e}1}+3)\mathrm{p}(\mathfrak{v}_1)}\zeta^{(b_{\mathfrak{e}2}+3)\mathrm{p}(\mathfrak{v}_2)}}.
\end{aligned}
\end{equation*}
\end{proof}

\begin{lem}\label{lem:DA^21_derivative_of_edge}
We have
\begin{equation*}
    \frac{\partial}{\partial (D^2A_1)}\mathrm{Cont}_{\Gamma}^{\mathrm{A}}(\mathfrak{e})
    =\frac{(-1)^{b_{\mathfrak{e}1}+b_{\mathfrak{e}2}}}{5L^2}
    \sum_{m=0}^2(-1)^m\frac{\widetilde{P}_{4,\mathrm{p}(\mathfrak{v}_1)}^{b_{\mathfrak{e}1}+m-2}\widetilde{P}_{4,\mathrm{p}(\mathfrak{v}_2)}^{b_{\mathfrak{e}2}-m}}{\zeta^{(b_{\mathfrak{e}1}+m+2)\mathrm{p}(\mathfrak{e}_1)}\zeta^{(b_{\mathfrak{e}2}-m+4)\mathrm{p}(\mathfrak{v}_2)}}.
\end{equation*}
\end{lem}
\begin{proof}
The strategy of proof is similar to that of Lemma \ref{lem:derivativeof_Pijk_wrtA2}. The only difference is that we use Lemma \ref{lem:derivativeof_Pijk_wrtDA1} instead of Lemma \ref{lem:derivativeof_Pijk_wrtA2} and shift one of the summations by $3$ rather than by $1$.
\end{proof}

\begin{thm}[The first holomorphic anomaly equation]\label{thm:HAE_wrt_A2_partial}
For $g\geq{2}$, we have
\begin{equation*}
\frac{C_{3}}{5L}\frac{\partial\mathcal{F}_{g}^{\left[\mathbb{C}^{5} / \mathbb{Z}_{5}\right]}}{\partial A_2}
=\frac{1}{2}\mathcal{F}_{g-1,2}^{\left[\mathbb{C}^5 / \mathbb{Z}_5\right]}\left(\phi_2,\phi_2\right)+\frac{1}{2}\sum_{i=1}^{g-1}\mathcal{F}_{g-i,1}^{\left[\mathbb{C}^5 / \mathbb{Z}_5\right]}\left(\phi_2\right)\mathcal{F}_{i,1}^{\left[\mathbb{C}^5 / \mathbb{Z}_5\right]}\left(\phi_2\right)
\end{equation*}
in $\mathbb{F}[C_1,C_2,C_3]$.
\end{thm}
\begin{proof}
Let $\Gamma\in\mathrm{G}_{g,0}^{\text{Dec}}(5)$ be a decorated graph and $\tilde{\mathfrak{e}}\in\mathrm{E}_{\Gamma}$ be an edge of $\Gamma$ connecting two vertices $\mathfrak{v}_1$ and $\mathfrak{v}_2$. After deleting the edge $\tilde{\mathfrak{e}}$, we obtain a new graph. (By deleting, we mean breaking the edge $\tilde{\mathfrak{e}}$ into two legs $\mathfrak{l}_{\tilde{\mathfrak{e}}}$ and $\mathfrak{l}'_{\tilde{\mathfrak{e}}}$.) There are two possibilities for the resulting graph after deletion of edge $\mathfrak{e}$:
\begin{itemize}
    \item[(i)] If it is connected, then we obtain an element of $\mathrm{G}_{g-1,2}^{\text{Dec}}(5)$, which we denote as $\Gamma_{\tilde{\mathfrak{e}}}^0$.
    
    \item[(ii)] If it is disconnected, then the resulting graph has two connected components, which we  denote as $\Gamma_{\tilde{\mathfrak{e}}}^1\in\mathrm{G}_{g_1,1}^{\text{Dec}}(5)$ and $\Gamma_{\tilde{\mathfrak{e}}}^2\in\mathrm{G}_{g_2,1}^{\text{Dec}}(5)$ where we have $g=g_1+g_2$.
\end{itemize}

By Proposition \ref{prop:contributions} and Lemma \ref{lem:A2_derivative_of_edge}, we observe that
\begin{equation*}
\begin{aligned}
\frac{\partial \mathrm{Cont}_{\Gamma}^{\mathrm{A}}(\tilde{\mathfrak{e}}) }{\partial A_2}
=&\frac{(-1)^{b_{\tilde{\mathfrak{e}}1}+b_{\tilde{\mathfrak{e}}2}}}{5} \frac{\widetilde{P}_{3,\mathrm{p}(\mathfrak{v}_1)}^{b_{\tilde{\mathfrak{e}}1}}\widetilde{P}_{3,\mathrm{p}(\mathfrak{v}_2)}^{b_{\tilde{\mathfrak{e}}2}}}{\zeta^{(b_{\tilde{\mathfrak{e}}1}+3)\mathrm{p}(\mathfrak{v}_1)}\zeta^{(b_{\tilde{\mathfrak{e}}2}+3)\mathrm{p}(\mathfrak{v}_2)}}\\
=&5\left(\frac{L^{3}}{K_{3}}\right)^2
\begin{cases}
\mathrm{Cont}_{\Gamma_{\tilde{\mathfrak{e}}}^0}^{\mathrm{A}}(\mathfrak{l}_{\tilde{\mathfrak{e}}})\mathrm{Cont}_{\Gamma_{\tilde{\mathfrak{e}}}^0}^{\mathrm{A}}(\mathfrak{l}'_{\tilde{\mathfrak{e}}})&\text{for the case (i)},\\
\mathrm{Cont}_{\Gamma_{\tilde{\mathfrak{e}}}^1}^{\mathrm{A}}(\mathfrak{l}_{\tilde{\mathfrak{e}}})\mathrm{Cont}_{\Gamma_{\tilde{\mathfrak{e}}}^2}^{\mathrm{A}}(\mathfrak{l}'_{\tilde{\mathfrak{e}}})&\text{for the case (ii)}
\end{cases}
\end{aligned}
\end{equation*}
with $\ell(\mathfrak{l}_{\tilde{\mathfrak{e}}})=\ell(\mathfrak{l}'_{\tilde{\mathfrak{e}}})=2$, i.e., with insertions $\phi_2$
.

By definition of $K_3$ and equation (\ref{eq:C12C22C3L5}), we also note that 
\begin{equation*}
\left(\frac{L^{3}}{K_{3}}\right)^2=\frac{L}{C_{3}}.
\end{equation*}
Then, for case (i), we easily see that we have
\begin{equation}\label{eq:2_insertions_graph_cont}
\begin{aligned}
\operatorname{Cont}_{\Gamma_{\tilde{\mathfrak{e}}}^0}\left(\phi_2,\phi_2\right)
=&\frac{1}{|\mathrm{Aut}(\Gamma_{\tilde{\mathfrak{e}}}^{0,\mathrm{St}})|} \sum_{\mathrm{A} \in \mathbb{Z}_{\geq 0}^{\mathrm{F}({\Gamma_{\tilde{\mathfrak{e}}}^0})}} \prod_{\mathfrak{v} \in \mathrm{V}_{\Gamma_{\tilde{\mathfrak{e}}}^0}} \mathrm{Cont}_{\Gamma_{\tilde{\mathfrak{e}}}^0}^{\mathrm{A}}(\mathfrak{v}) \prod_{\mathfrak{e} \in \mathrm{E}_{\Gamma_{\tilde{\mathfrak{e}}}^0}} \mathrm{Cont}_{\Gamma_{\tilde{\mathfrak{e}}}^0}^{\mathrm{A}}(\mathfrak{e})\prod_{\mathfrak{l} \in \mathrm{L}_{\Gamma_{\tilde{\mathfrak{e}}}^0}} \mathrm{Cont}_{\Gamma_{\tilde{\mathfrak{e}}}^0}^{\mathrm{A}}(\mathfrak{l})\\
=&\frac{1}{|\mathrm{Aut}(\Gamma_{\tilde{\mathfrak{e}}}^{0,\mathrm{St}})|} \sum_{\mathrm{A} \in \mathbb{Z}_{\geq 0}^{\mathrm{F}(\Gamma)}}\frac{C_3}{5L}\frac{\partial \mathrm{Cont}_{\Gamma}^{\mathrm{A}}(\tilde{\mathfrak{e}}) }{\partial A_2} \prod_{\mathfrak{v} \in \mathrm{V}_{\Gamma}} \mathrm{Cont}_{\Gamma}^{\mathrm{A}}(\mathfrak{v}) \prod_{\substack{
\mathfrak{e} \in \mathrm{E}_{\Gamma} \\ \mathfrak{e}\neq \tilde{\mathfrak{e}}}} \mathrm{Cont}_{\Gamma}^{\mathrm{A}}(\mathfrak{e}).
\end{aligned}
\end{equation}
Similarly, for case (ii), we observe the following
\begin{equation}\label{eq:prod_of_1_insertions_graph_conts}
\begin{aligned}
\operatorname{Cont}_{\Gamma_{\tilde{\mathfrak{e}}}^1}\left(\phi_{2}\right)&\operatorname{Cont}_{\Gamma_{\tilde{\mathfrak{e}}}^2}\left(\phi_{2}\right)\\
=&\frac{1}{|\mathrm{Aut}(\Gamma_{\tilde{\mathfrak{e}}}^{1,\mathrm{St}})|} \sum_{\mathrm{A} \in \mathbb{Z}_{\geq 0}^{\mathrm{F}({\Gamma_{\tilde{\mathfrak{e}}}^1})}} \mathrm{Cont}_{\Gamma_{\tilde{\mathfrak{e}}}^1}^{\mathrm{A}}(\mathfrak{l}_{\tilde{\mathfrak{e}}})\prod_{\mathfrak{v} \in \mathrm{V}_{\Gamma_{\tilde{\mathfrak{e}}}^1}} \mathrm{Cont}_{\Gamma_{\tilde{\mathfrak{e}}}^1}^{\mathrm{A}}(\mathfrak{v}) \prod_{\mathfrak{e} \in \mathrm{E}_{\Gamma_{\tilde{\mathfrak{e}}}^1}} \mathrm{Cont}_{\Gamma_{\tilde{\mathfrak{e}}}^1}^{\mathrm{A}}(\mathfrak{e})\\
\times&\frac{1}{|\mathrm{Aut}(\Gamma_{\tilde{\mathfrak{e}}}^{2,\mathrm{St}})|} \sum_{\mathrm{A} \in \mathbb{Z}_{\geq 0}^{\mathrm{F}({\Gamma_{\tilde{\mathfrak{e}}}^2})}} \mathrm{Cont}_{\Gamma_{\tilde{\mathfrak{e}}}^2}^{\mathrm{A}}(\mathfrak{l}'_{\tilde{\mathfrak{e}}}) \prod_{v \in \mathrm{V}_{\Gamma_{\tilde{\mathfrak{e}}}^2}} \mathrm{Cont}_{\Gamma_{\tilde{\mathfrak{e}}}^2}^{\mathrm{A}}(\mathfrak{v}) \prod_{\mathfrak{e} \in \mathrm{E}_{\Gamma_{\tilde{\mathfrak{e}}}^2}} \mathrm{Cont}_{\Gamma_{\tilde{\mathfrak{e}}}^2}^{\mathrm{A}}(\mathfrak{e})\\
=&\frac{1}{|\mathrm{Aut}(\Gamma_{\tilde{\mathfrak{e}}}^{1,\mathrm{St}})||\mathrm{Aut}(\Gamma_{\tilde{\mathfrak{e}}}^{,\mathrm{St}})|} \sum_{\mathrm{A} \in \mathbb{Z}_{\geq 0}^{\mathrm{F}(\Gamma)}}\frac{C_{3}}{5L}\frac{\partial \mathrm{Cont}_{\Gamma}^{\mathrm{A}}(\tilde{\mathfrak{e}}) }{\partial A_2} \prod_{\mathfrak{v} \in \mathrm{V}_{\Gamma}} \mathrm{Cont}_{\Gamma}^{\mathrm{A}}(\mathfrak{v}) \prod_{\substack{
\mathfrak{e} \in \mathrm{E}_{\Gamma} \\ \mathfrak{e}\neq \tilde{\mathfrak{e}}}} \mathrm{Cont}_{\Gamma}^{\mathrm{A}}(\mathfrak{e}).
\end{aligned}
\end{equation}
By Lemma \ref{lem:P0jk_is_in_CLplusminus} and Theorem \ref{thm:vertex_edge_leg_Rings}, we have the following vanishing result:
\begin{equation*}
\frac{\partial \mathrm{Cont}_{\Gamma}^{\mathrm{A}}(\mathfrak{v}) }{\partial A_2}=0
\end{equation*}
for any vertex $\mathfrak{v}\in\mathrm{V}_{\Gamma} $.

Then, this vanishing result gives us the following
\begin{equation*}
\begin{aligned}
\frac{\partial \mathrm{Cont}_{\Gamma}}{\partial A_2}
=&\frac{1}{|\mathrm{Aut}(\Gamma^{\mathrm{St}})|} \sum_{\mathrm{A} \in \mathbb{Z}_{\geq 0}^{\mathrm{F}(\Gamma)}} \prod_{\mathfrak{v} \in \mathrm{V}_{\Gamma}} \mathrm{Cont}_{\Gamma}^{\mathrm{A}}(\mathfrak{v})\frac{\partial}{\partial A_2}\left(\prod_{\mathfrak{e}\in \mathrm{E}_{\Gamma}} \mathrm{Cont}_{\Gamma}^{\mathrm{A}}(\mathfrak{e})\right)\\
=&\sum_{\tilde{\mathfrak{e}}\in \mathrm{E}_{\Gamma}}\frac{1}{|\mathrm{Aut}(\Gamma^{\mathrm{St}})|} \sum_{\mathrm{A} \in \mathbb{Z}_{\geq 0}^{\mathrm{F}(\Gamma)}}\frac{\partial \mathrm{Cont}_{\Gamma}^{\mathrm{A}}(\tilde{\mathfrak{e}}) }{\partial A_2} \prod_{\mathfrak{v} \in \mathrm{V}_{\Gamma}} \mathrm{Cont}_{\Gamma}^{\mathrm{A}}(\mathfrak{v})\prod_{\substack{
\mathfrak{e} \in \mathrm{E}_{\Gamma} \\ \mathfrak{e}\neq \tilde{\mathfrak{e}}}} \mathrm{Cont}_{\Gamma}^{\mathrm{A}}(\mathfrak{e}).
\end{aligned}
\end{equation*}
So, we have
\begin{equation}\label{eq:A2_Derivative_of_Gamma_cont}
\frac{C_{3}}{5L}\frac{\partial \mathrm{Cont}_{\Gamma}}{\partial A_2}
=\sum_{\tilde{\mathfrak{e}}\in \mathrm{E}_{\Gamma}}\frac{1}{|\mathrm{Aut}(\Gamma^{\mathrm{St}})|} \sum_{\mathrm{A} \in \mathbb{Z}_{\geq 0}^{\mathrm{F}(\Gamma)}}\frac{C_{3}}{5L}\frac{\partial \mathrm{Cont}_{\Gamma}^{\mathrm{A}}(\tilde{\mathfrak{e}}) }{\partial A_2} \prod_{\mathfrak{v} \in \mathrm{V}_{\Gamma}} \mathrm{Cont}_{\Gamma}^{\mathrm{A}}(\mathfrak{v})\prod_{\substack{
\mathfrak{e} \in \mathrm{E}_{\Gamma} \\ \mathfrak{e}\neq \tilde{\mathfrak{e}}}} \mathrm{Cont}_{\Gamma}^{\mathrm{A}}(\mathfrak{e}).
\end{equation}
Then, summing equation (\ref{eq:2_insertions_graph_cont}) and equation (\ref{eq:prod_of_1_insertions_graph_conts}) over all decorated stable graphs $\Gamma_{\tilde{\mathfrak{e}}}^0$ and $(\Gamma_{\tilde{\mathfrak{e}}}^1,\Gamma_{\tilde{\mathfrak{e}}}^2)$
we obtain
\begin{equation*}
\left\langle\left\langle\phi_{2},\phi_{2}\right\rangle\right\rangle_{g-1,2}^{\left[\mathbb{C}^5 / \mathbb{Z}_5\right]}\quad\text{and}\quad\sum_{i=1}^{g-1}\left\langle\left\langle\phi_{2}\right\rangle\right\rangle_{g-i,1}^{\left[\mathbb{C}^5 / \mathbb{Z}_5\right]}\left\langle\left\langle\phi_{2}\right\rangle\right\rangle_{i,1}^{\left[\mathbb{C}^5 / \mathbb{Z}_5\right]}
\end{equation*}
respectively. Then, by equation (\ref{eq:A2_Derivative_of_Gamma_cont}) and an analysis of automorphism factors appearing in equations (\ref{eq:2_insertions_graph_cont}), (\ref{eq:prod_of_1_insertions_graph_conts}), we conclude that
\begin{equation*}
2\frac{C_{3}}{5L}\frac{\partial}{\partial A_2}\left\langle\left\langle\right\rangle\right\rangle_{g}^{\left[\mathbb{C}^5 / \mathbb{Z}_5\right]}
=\left\langle\left\langle\phi_{2},\phi_{2}\right\rangle\right\rangle_{g-1,2}^{\left[\mathbb{C}^5 / \mathbb{Z}_5\right]}+\sum_{i=1}^{g-1}\left\langle\left\langle\phi_{2}\right\rangle\right\rangle_{g-i,1}^{\left[\mathbb{C}^5 / \mathbb{Z}_5\right]}\left\langle\left\langle\phi_{2}\right\rangle\right\rangle_{i,1}^{\left[\mathbb{C}^5 / \mathbb{Z}_5\right]}
\end{equation*}
after summing over all decorated stable graphs $\Gamma$. The reason we have $2$ in front of the left hand side is due to not having a canonical order of labelings of each of the legs $\mathfrak{l}_{\tilde{\mathfrak{e}}}$ and $\mathfrak{l}'_{\tilde{\mathfrak{e}}}$ for case (i) and double counting  of different genera of connected components for case (ii). This completes the proof. 
\end{proof}

Let $\pi$ be the morphism to the moduli space of stable curves determined by the domain, $$\pi:\overline{M}_{g, k}^{\mathrm{orb}}\left(\left[\mathbb{C}^{5} / \mathbb{Z}_{5}\right], 0\right)\rightarrow \overline{M}_{g, k}.$$
We define the Gromov-Witten potentials with certain types of ancestor insertions by
\begin{equation*}
\mathcal{F}_{g, n}^{\left[\mathbb{C}^{5} / \mathbb{Z}_{5}\right]}\left(\phi_{c_{1}}\psi_1^{a_1}, \ldots, \phi_{c_{n}}\psi_n^{a_n}\right)=\sum_{d=0}^{\infty} \frac{\Theta^{d}}{d !} \int_{\left[\overline{M}_{g, n+d}^{\mathrm{orb}}\left(\left[\mathbb{C}^{5} / \mathbb{Z}_{5}\right], 0\right)\right]^{v i r}} \prod_{i=1}^{n} \mathrm{ev}_{i}^{*}\left(\phi_{c_{i}})\pi^{*}(\psi_i^{a_i}\right) \prod_{i=n+1}^{n+d} \mathrm{ev}_{i}^{*}\left(\phi_{1}\right).
\end{equation*}

These Gromov-Witten potentials can also be written as a sum over decorated stables graphs as in Proposition \ref{prop:contributions}. The only difference with Proposition \ref{prop:contributions} happens at vertex contributions and it is the shift in the powers of $\psi$ classes according to the insertions. The rest is the same. Hence, we again conclude that
$$\mathcal{F}_{g, n}^{\left[\mathbb{C}^{5} / \mathbb{Z}_{5}\right]}\left(\phi_{c_{1}}\psi_1^{a_1}, \ldots, \phi_{c_{n}}\psi_n^{a_n}\right)\in\mathbb{F}[C_1,C_2,C_3].$$
Now, we are ready to state another holomorphic anomaly equation.

\begin{thm}[The second holomorphic anomaly equation]\label{thm:2nd_HAE}
For $g\geq{2}$, we have
\begin{equation*}
\begin{aligned}
\frac{C_2^2C_3}{5L^3}\frac{\partial\mathcal{F}_{g}^{\left[\mathbb{C}^{5} / \mathbb{Z}_{5}\right]}}{\partial (D^2A_1)}
=&\mathcal{F}_{g-1,2}^{\left[\mathbb{C}^5 / \mathbb{Z}_5\right]}\left(\phi_1\psi_1^2,\phi_1\right)+\sum_{i=1}^{g-1}\mathcal{F}_{g-i,1}^{\left[\mathbb{C}^5 / \mathbb{Z}_5\right]}\left(\phi_1\psi_1^2\right)\mathcal{F}_{i,1}^{\left[\mathbb{C}^5 / \mathbb{Z}_5\right]}\left(\phi_1\right)\\
&-\frac{1}{2}\mathcal{F}_{g-1,2}^{\left[\mathbb{C}^5 / \mathbb{Z}_5\right]}\left(\phi_1\psi_1,\phi_1\psi_2\right) -\frac{1}{2}\sum_{i=1}^{g-1}\mathcal{F}_{g-i,1}^{\left[\mathbb{C}^5 / \mathbb{Z}_5\right]}\left(\phi_1\psi_1\right)\mathcal{F}_{i,1}^{\left[\mathbb{C}^5 / \mathbb{Z}_5\right]}\left(\phi_1\psi_1\right)  
\end{aligned}
\end{equation*}
in $\mathbb{F}[C_1,C_2,C_3]$.
\end{thm}

\begin{proof}
The proof is similar to that of Theorem \ref{thm:HAE_wrt_A2_partial} with some technical difference. Instead of giving full details, this time we point out these different technicalities. Throughout the proof, let $\Gamma$, $\tilde{\mathfrak{e}}$, $\mathfrak{v}_1$, $\mathfrak{v}_2$, $\mathfrak{l}_{\tilde{\mathfrak{e}}}$, $\mathfrak{l}'_{\tilde{\mathfrak{e}}}$, $\Gamma_{\tilde{\mathfrak{e}}}^0$, $\Gamma_{\tilde{\mathfrak{e}}}^1$, $\Gamma_{\tilde{\mathfrak{e}}}^2$, ``case (i)'' and ``case (ii)'' be as in the proof of Theorem \ref{thm:HAE_wrt_A2_partial}.

By Proposition \ref{prop:contributions} and Lemma \ref{lem:DA^21_derivative_of_edge}, we have
\begin{equation}\label{eq:partial_wrt_DA21_in_proof}
\begin{aligned}
    \frac{\partial}{\partial (D^2A_1)}\mathrm{Cont}_{\Gamma}^{\mathrm{A}}(\mathfrak{e})
    =&\frac{(-1)^{b_{\tilde{\mathfrak{e}}1}+b_{\tilde{\mathfrak{e}}2}}}{5L^2}
    \frac{\widetilde{P}_{4,\mathrm{p}(\mathfrak{v}_1)}^{b_{\tilde{\mathfrak{e}}1}-2}\widetilde{P}_{4,\mathrm{p}(\mathfrak{v}_2)}^{b_{\tilde{\mathfrak{e}}2}}}{\zeta^{(b_{\tilde{\mathfrak{e}}1}+2)\mathrm{p}(\tilde{\mathfrak{e}}_1)}\zeta^{(b_{\tilde{\mathfrak{e}}2}+4)\mathrm{p}(\mathfrak{v}_2)}}\\
    &-\frac{(-1)^{b_{\tilde{\mathfrak{e}}1}+b_{\tilde{\mathfrak{e}}2}}}{5L^2}
    \frac{\widetilde{P}_{4,\mathrm{p}(\mathfrak{v}_1)}^{b_{\tilde{\mathfrak{e}}1}-1}\widetilde{P}_{4,\mathrm{p}(\mathfrak{v}_2)}^{b_{\tilde{\mathfrak{e}}2}-1}}{\zeta^{(b_{\tilde{\mathfrak{e}}1}+3)\mathrm{p}(\tilde{\mathfrak{e}}_1)}\zeta^{(b_{\tilde{\mathfrak{e}}2}+3)\mathrm{p}(\mathfrak{v}_2)}}\\
    &+\frac{(-1)^{b_{\tilde{\mathfrak{e}}1}+b_{\tilde{\mathfrak{e}}2}}}{5L^2}
    \frac{\widetilde{P}_{4,\mathrm{p}(\mathfrak{v}_1)}^{b_{\tilde{\mathfrak{e}}1}}\widetilde{P}_{4,\mathrm{p}(\mathfrak{v}_2)}^{b_{\tilde{\mathfrak{e}}2}-2}}{\zeta^{(b_{\tilde{\mathfrak{e}}1}+4)\mathrm{p}(\tilde{\mathfrak{e}}_1)}\zeta^{(b_{\tilde{\mathfrak{e}}2}+2)\mathrm{p}(\mathfrak{v}_2)}},
\end{aligned}
\end{equation}
where right hand side of this equation is equal to
\begin{equation*}
5\left(\frac{L^{4}}{K_{4}}\right)^2\left(
\mathrm{Cont}_{\Gamma_{\tilde{\mathfrak{e}}}^0}^{\mathrm{A}_{b_{\tilde{\mathfrak{e}}1}-2}}(\mathfrak{l}_{\tilde{\mathfrak{e}}})\mathrm{Cont}_{\Gamma_{\tilde{\tilde{\mathfrak{e}}}}^0}^{\mathrm{A}}(\mathfrak{l}'_{\tilde{\mathfrak{e}}})-\mathrm{Cont}_{\Gamma_{\tilde{\mathfrak{e}}}^0}^{\mathrm{A}_{b_{\tilde{\mathfrak{e}}1}-1}}(\mathfrak{l}_{\tilde{\mathfrak{e}}})\mathrm{Cont}_{\Gamma_{\tilde{\mathfrak{e}}}^0}^{\mathrm{A}_{b_{\tilde{\mathfrak{e}}2}-1}}(\mathfrak{l}'_{\tilde{\mathfrak{e}}})+\mathrm{Cont}_{\Gamma_{\tilde{\mathfrak{e}}}^0}^{\mathrm{A}}(\mathfrak{l}_{\tilde{\mathfrak{e}}})\mathrm{Cont}_{\Gamma_{\tilde{\mathfrak{e}}}^0}^{\mathrm{A}_{b_{\tilde{\mathfrak{e}}2}-2}}(\mathfrak{l}'_{\tilde{\mathfrak{e}}})\right)
\end{equation*}
for the case (i), and it is equal to
\begin{equation*}
5\left(\frac{L^{4}}{K_{4}}\right)^2\left(\mathrm{Cont}_{\Gamma_{\tilde{\mathfrak{e}}}^1}^{\mathrm{A}_{b_{\tilde{\mathfrak{e}}1}-2}}(\mathfrak{l}_{\tilde{\mathfrak{e}}})\mathrm{Cont}_{\Gamma_{\tilde{\mathfrak{e}}}^2}^{\mathrm{A}}(\mathfrak{l}'_{\tilde{\mathfrak{e}}})
-\mathrm{Cont}_{\Gamma_{\tilde{\mathfrak{e}}}^1}^{\mathrm{A}_{b_{\tilde{\mathfrak{e}}1}-1}}(\mathfrak{l}_{\tilde{\mathfrak{e}}})\mathrm{Cont}_{\Gamma_{\tilde{\mathfrak{e}}}^2}^{\mathrm{A}_{b_{\tilde{\mathfrak{e}}2}-1}}(\mathfrak{l}'_{\tilde{\mathfrak{e}}})
+\mathrm{Cont}_{\Gamma_{\tilde{\mathfrak{e}}}^1}^{\mathrm{A}}(\mathfrak{l}_{\tilde{\mathfrak{e}}})\mathrm{Cont}_{\Gamma_{\tilde{\mathfrak{e}}}^2}^{\mathrm{A}_{b_{\tilde{\mathfrak{e}}2}-2}}(\mathfrak{l}'_{\tilde{\mathfrak{e}}})
\right)
\end{equation*}
for the case (ii), with $\ell(\mathfrak{l}_{\tilde{\mathfrak{e}}})=\ell(\mathfrak{l}'_{\tilde{\mathfrak{e}}})=1$. Here, by $\mathrm{A}_{b_{\tilde{\mathfrak{e}}s}-r}$ we mean the flag value $b_{\tilde{\mathfrak{e}}s}$ is shifted by $r$ in $\mathrm{A}\in\mathbb{Z}_{\geq 0}^{\mathrm{F}(\bullet)}$ where $\bullet$ is $\Gamma_{\tilde{\mathfrak{e}}}^0$, $\Gamma_{\tilde{\mathfrak{e}}}^1$ or $\Gamma_{\tilde{\mathfrak{e}}}^2$. Also note that, by definition of $K_4$ and equation (\ref{eq:C12C22C3L5}), we have
\begin{equation*}
\left(\frac{L^{4}}{K_{4}}\right)^2=\frac{L^3}{C_2^2C_3}.
\end{equation*}

By shifting the flag values by $r$ in equation (\ref{eq:partial_wrt_DA21_in_proof}), we get
\begin{equation}\label{eq:alternate_sum_after_derivative}
\frac{(-1)^{b_{\tilde{\mathfrak{e}}1}+b_{\tilde{\mathfrak{e}}2}}}{5L^2}\left(
    \frac{\widetilde{P}_{4,\mathrm{p}(\mathfrak{v}_1)}^{b_{\tilde{\mathfrak{e}}1}}\widetilde{P}_{4,\mathrm{p}(\mathfrak{v}_2)}^{b_{\tilde{\mathfrak{e}}2}}}{\zeta^{(b_{\tilde{\mathfrak{e}}1}+4)\mathrm{p}(\tilde{\mathfrak{e}}_1)}\zeta^{(b_{\tilde{\mathfrak{e}}2}+4)\mathrm{p}(\mathfrak{v}_2)}}
    -\frac{\widetilde{P}_{4,\mathrm{p}(\mathfrak{v}_1)}^{b_{\tilde{\mathfrak{e}}1}}\widetilde{P}_{4,\mathrm{p}(\mathfrak{v}_2)}^{b_{\tilde{\mathfrak{e}}2}}}{\zeta^{(b_{\tilde{\mathfrak{e}}1}+4)\mathrm{p}(\tilde{\mathfrak{e}}_1)}\zeta^{(b_{\tilde{\mathfrak{e}}2}+4)\mathrm{p}(\mathfrak{v}_2)}}
    +\frac{\widetilde{P}_{4,\mathrm{p}(\mathfrak{v}_1)}^{b_{\tilde{\mathfrak{e}}1}}\widetilde{P}_{4,\mathrm{p}(\mathfrak{v}_2)}^{b_{\tilde{\mathfrak{e}}2}}}{\zeta^{(b_{\tilde{\mathfrak{e}}1}+4)\mathrm{p}(\tilde{\mathfrak{e}}_1)}\zeta^{(b_{\tilde{\mathfrak{e}}2}+4)\mathrm{p}(\mathfrak{v}_2)}}
    \right),
\end{equation}
where $r$ is $2$ in the first term for $b_{\tilde{\mathfrak{e}}1}$, $r$ is $1$ in the second term for $b_{\tilde{\mathfrak{e}}1}$, $b_{\tilde{\mathfrak{e}}2}$, and $r$ is $2$ in the last term for $b_{\tilde{\mathfrak{e}}2}$. If we make these shifts after taking the derivative of a graph contribution with respect to $D^2A_1$, they will also effect the contributions of vertices that $\tilde{\mathfrak{e}}$ is attached to. More precisely, they will also shift the powers of the corresponding $\psi$ classes appearing in the contributions of the vertices that $\tilde{\mathfrak{e}}$ is attached to.

If we adapt of the proof of Theorem \ref{thm:HAE_wrt_A2_partial} at this point, then the first and last terms in equation (\ref{eq:alternate_sum_after_derivative}) explain the terms
$$\mathcal{F}_{g-1,2}^{\left[\mathbb{C}^5 / \mathbb{Z}_5\right]}\left(\phi_1\psi_1^2,\phi_1\right)+\sum_{i=1}^{g-1}\mathcal{F}_{g-i,1}^{\left[\mathbb{C}^5 / \mathbb{Z}_5\right]}\left(\phi_1\psi_1^2\right)\mathcal{F}_{i,1}^{\left[\mathbb{C}^5 / \mathbb{Z}_5\right]}\left(\phi_1\right)$$
in the holomorhic anomaly equation and the middle term in equation (\ref{eq:alternate_sum_after_derivative}) explains
the terms
$$-\frac{1}{2}\mathcal{F}_{g-1,2}^{\left[\mathbb{C}^5 / \mathbb{Z}_5\right]}\left(\phi_1\psi_1,\phi_1\psi_2\right) -\frac{1}{2}\sum_{i=1}^{g-1}\mathcal{F}_{g-i,1}^{\left[\mathbb{C}^5 / \mathbb{Z}_5\right]}\left(\phi_1\psi_1\right)\mathcal{F}_{i,1}^{\left[\mathbb{C}^5 / \mathbb{Z}_5\right]}\left(\phi_1\psi_1\right)$$
in the holomorphic anomaly equation.
\end{proof}

\end{document}